\documentclass{article}
\usepackage{fontenc}
\usepackage{amsmath}
\usepackage{amssymb}
\usepackage{amsfonts}
\usepackage{amsthm}

\newtheorem{thm}{Theorem}
\newtheorem{prop}[thm]{Proposition}
\newtheorem{cor}[thm]{Corollary}

\begin{document}

\title{A Direct Evaluation of the Periods of the Weierstrass Zeta Function}

\author{Shaul Zemel\thanks{The initial stage of this research has been carried
out as part of my Ph.D. thesis work at the Hebrew University of Jerusalem,
Israel. The final stage of this work was supported by the Minerva Fellowship
(Max-Planck-Gesellschaft).}}

\maketitle

\section*{Introduction}

The Weierstrass Zeta function \[Z_{L}(z)=\frac{1}{z}+\sum_{0\neq\lambda \in
L}\bigg(\frac{1}{z-\lambda}+\frac{1}{\lambda}+\frac{z}{\lambda^{2}}\bigg)\] of a
rank 2 lattice $L\subseteq\mathbb{C}$ is very important for the theory of
elliptic functions. Though not elliptic itself, (minus) its derivative
\[\wp_{L}(z)=\frac{1}{z^{2}}+\sum_{0\neq\lambda \in
L}\bigg(\frac{1}{(z-\lambda)^{2}}-\frac{1}{\lambda^{2}}\bigg)\] is elliptic, and
together with $\wp_{L}'$ it generates the field of functions which are elliptic
with respect to $L$. The Zeta function $Z_{L}$ also provides the relation
between the Weierstrass $\sigma$ function, a theta function of the lattice $L$,
and the elliptic function $\wp_{L}$, since $Z_{L}$ is the logarithmic derivative
of the $\sigma$ function and $Z_{L}'=-\wp_{L}$. Moreover, just like $\wp_{L}$ is
used in order to construct Eisenstein series of weight 2 (and from its
derivatives one obtains Eisenstein series of higher weights), the function
$Z_{L}$ is used in order to construct Eisenstein series of weight 1 (see Chapter
4 of \cite{[DS]}). In fact, the results of the present paper have been obtained
during the author's study of this construction of weight 1 Eisenstein series,
appearing in \cite{[DS]}.

Since $\wp_{L}$ is elliptic, the Zeta function $Z_{L}$ has a difference function
\[\eta_{L}:L\to\mathbb{C},\quad \eta_{L}(\lambda)=Z_{L}(z+\lambda)-Z_{L}(z),\]
where the latter expression is a constant independent of $z$. The map $\eta_{L}$
is a homomorphism, and can be considered as the period map of $\wp$ with respect
to $L=H_{1}(\mathbb{C}/L,\mathbb{Z})$. These constants $\eta_{L}(\lambda)$ have
various applications in the theory of elliptic functions: In particular, one of
these constants shows up in the multiplier of the Fourier expansion of the sigma
function. In all the references known to the author (see, for example, Chapter
18 of \cite{[L]}), the evaluation of $\eta_{L}$ requires various indirect tools
(contour integration, for example). The purpose of this work is to introduce a
specific order of summation for the series defining the zeta function which
enables us to obtain the values of $\eta_{L}$ directly, without the use of any
additional tools. The same technique applies for obtaining the ellipticity of
$\wp_{L}$ and its derivatives, though here the standard proof is also short and
simple.

The values of $\eta_{L}$ for $L=L_{\tau}$ involve the value of the Eisenstein
series $G_{2}$ at $\tau\in\mathcal{H}$, defined by
\[G_{2}(\tau)=\sum_{c\in\mathbb{Z}}\sum_{d,(c,d)\neq(0,0)}\frac{1}{(c\tau+d)^{2}
}\] (inner summation on $d$, then on $c$). One knows that $G_{2}$ is
quasi-modular of weight 2, but the proof must address the problem that the
defining series for $G_{2}$ converges only conditionally. Here we show how the
full quasi-modular behavior of $G_{2}$ follows immediately from its relation
with $\eta_{L}$ and from the homogeneity property of the latter, saving even the
need to verify that the quasi-modular action is an action.

In Section \ref{wpandder} we illustrate the application of the method for the
simple case of $\wp$ and its derivatives. Section \ref{Zeta} considers the
Weierstrass Zeta function and shows how this method evaluates $\eta_{L}$
directly. Finally, in Section \ref{G2} we use these results to derive the
quasi-modularity of $G_{2}$ in a simple and transparent manner.

I wish to thank J. Shurman, with whom I had a long and enlightening
correspondence while I was studying modular forms from \cite{[DS]}. Thanks are
also due to J. Bruinier and E. Freitag, who read this paper and offered useful
advice and clarifications.

\section{A Simple Illustration: $\wp$ and its Derivatives}\label{wpandder}

Let $\mathbb{C}$ denote the field of complex numbers, and for $w\in\mathbb{C}$,
$\Im w$ denotes the imaginary part of $w$. We denote the upper half plane
$\{\tau\in\mathbb{C}|\Im\tau>0\}$ by $\mathcal{H}$, and a \emph{lattice} $L$ in
$\mathbb{C}$ is a rank 2 discrete subgroup of the additive group of
$\mathbb{C}$.

The analysis becomes much simpler when one of generators of the lattice
$L\subseteq\mathbb{C}$ is 1. For an element $\tau\in\mathcal{H}$, we denote the
lattice $\mathbb{Z}\tau\oplus\mathbb{Z}$ by $L_{\tau}$, and the index $L_{\tau}$
is classically replaced by $\tau$. It is evident that every lattice $L$ is
$\alpha L_{\tau}$ for some $\alpha\in\mathbb{C}$ and $\tau\in\mathcal{H}$ (but
not uniquely).

\medskip

We start by presenting our argument for the simplest case of the derivatives of
$\wp$. For $k\geq3$ we have that the $(k-2)$th derivative of $\wp_{L}$ is
\[\wp_{L}^{(k-2)}(z)=(-1)^{k}(k-1)!\sum_{\lambda \in
L}\frac{1}{(z-\lambda)^{k}},\] hence for $L=L_{\tau}$ we get
\[\wp_{\tau}^{(k-2)}(z)=(-1)^{k}(k-1)!\sum_{(c,d)\in\mathbb{Z}^2}\frac{1}{
(z-(c\tau+d))^k}.\] The only identities we shall need are the classical equality
\begin{equation}
\frac{1}{w}+\sum_{d=1}^{\infty}\bigg(\frac{1}{w+d}+\frac{1}{w-d}
\bigg)=\pi\cot\pi w=-\pi i-2\pi i\sum_{m=1}^{\infty}\mathbf{e}(mw) \label{ctgH}
\end{equation}
and its derivatives
\begin{equation}
\sum_{d=-\infty}^{\infty}\frac{1}{(w+d)^{k}}=\frac{(-2\pi
i)^{k}}{(k-1)!}\sum_{m=1}^{\infty}m^{k-1}\mathbf{e}(mw),
\label{ctgHder}
\end{equation}
both valid for $w\in\mathcal{H}$ with $\mathbf{e}(\sigma)=e^{2\pi i\sigma}$ for
all $\sigma\in\mathbb{C}$ (see, for example, Equations (1.1) and (1.2) of
Chapter 1 of \cite{[DS]}---note that the summation in Equation (1.1) there
begins with $m=0$ while we start with $m=1$). The left equality in Equation
\eqref{ctgH} follows by taking the logarithmic derivative (at $z=w$) of the
product expansion for the sine function, \[\sin\pi z=\pi
z\prod_{d=1}^{\infty}\bigg(1-\frac{z^{2}}{d^{2}}\bigg).\] The right equality
there is obtained by expanding $\pi\cot\pi w=\pi
i\big(-1-\frac{2\mathbf{e}(w)}{1-\mathbf{e}(w)}\big)$ as a geometric series. We
shall later obtain, however, expressions involving the left hand side of
Equations \eqref{ctgH} and \eqref{ctgHder} for $w\not\in\mathcal{H}$. For
$w\in\overline{\mathcal{H}}$ we write $w=-(-w)$, or alternatively use the
geometric expansion of $\pi\cot\pi w$ as $\pi
i\big(1+\frac{2\mathbf{e}(-w)}{1-\mathbf{e}(-w)}\big)$. Then Equation
\eqref{ctgH} takes the form
\begin{equation}
\frac{1}{w}+\sum_{d=1}^{\infty}\bigg(\frac{1}{w+d}+\frac{1}{w-d}
\bigg)=\pi\cot\pi w=+\pi i+2\pi i\sum_{m=1}^{\infty}\mathbf{e}(-mw)
\label{ctgHbar}
\end{equation}
and the corresponding Equation \eqref{ctgHder} is obtained by differentiation:
\begin{equation}
\sum_{d=-\infty}^{\infty}\frac{1}{(w+d)^{k}}=\frac{(+2\pi
i)^{k}}{(k-1)!}\sum_{m=1}^{\infty}m^{k-1}\mathbf{e}(-mw).
\label{ctgHbarder}
\end{equation}

We wish to substitute the value $w=z-c\tau$ in these formulae, where $\tau$ is
the index of the lattice $L_{\tau}$ and $z$ is the argument of the function we
investigate. We avoid the poles of the functions by considering $z \not\in
L_{\tau}$, which still allows $w$ to be real, though not integral. For a real
number $x$ we denote its lower integral value $\lfloor x
\rfloor=\max\{n\in\mathbb{Z}|n \leq x\}$, its upper integral value $\lceil x
\rceil=\min\{n\in\mathbb{Z}|n \geq x\}$, and its fractional part
$\{x\}=x-\lfloor x \rfloor$, the latter being the unique number $0 \leq a<1$
which lies in $x+\mathbb{Z}$. Then for real $w$ we use Equation \eqref{ctgH} or
\eqref{ctgHbar} in the form
\begin{equation}
\frac{1}{w}+\sum_{d=1}^{\infty}\bigg(\frac{1}{w+d}+\frac{1}{w-d}
\bigg)=\pi\cot\pi w=\pi\cot\pi a,\qquad a=\{w\} \label{ctgR}
\end{equation}
(without the Fourier expansion). In Equation \eqref{ctgHder} or
\eqref{ctgHbarder} we cannot use the Fourier expansion as well, but decomposing
the left hand side of these equations to $d>-w$ and $d<-w$ (again, we do not
have a term with $d=-w$ since we assume $w\not\in\mathbb{Z}$) gives
\begin{equation}
\sum_{d=-\infty}^{\infty}\frac{1}{(w+d)^{k}}=\zeta(k,a)+(-1)^{k}\zeta(k,1-a),
\qquad a=\{w\}. \label{ctgRder}
\end{equation}
Here $\zeta(s,v)$ denotes the Hurwitz zeta function, defined for $\Re s>1$ and
$v>0$ by the series $\sum_{n=0}^{\infty}\frac{1}{(n+v)^{s}}$.

\medskip

Since the sum defining $\wp_{\tau}^{(k-2)}(z)$ converges absolutely, we carry
out the summation first on $d$ and then on $c$. We use Equation \eqref{ctgHder}
for $c<\frac{\Im z}{\Im\tau}$, \eqref{ctgHbarder} for $c>\frac{\Im z}{\Im\tau}$,
and in case $\frac{\Im z}{\Im\tau}$ is an integer we use Equation
\eqref{ctgRder} for $c=\frac{\Im z}{\Im\tau}$ (we assume $z \not\in L_{\tau}$),
where $w=z-c\tau$. We use the classical notation $q_{\tau}=\mathbf{e}(\tau)$ and
$q_{z}=\mathbf{e}(z)$, hence $\mathbf{e}(mw)=q_{z}^{m}q_{\tau}^{-cm}$ and
$\mathbf{e}(-mw)=q_{\tau}^{cm}q_{z}^{-m}$ for any $m\in\mathbb{N}$. These
substitutions yield that
\[\frac{(-1)^{k}}{(k-1)!}\wp_{\tau}^{(k-2)}(z)=\sum_{c}\sum_{d}\frac{1}{
((z-c\tau)-d)^{k}}\] equals \[\frac{(-2\pi i)^{k}}{(k-1)!}\sum_{c<\frac{\Im
z}{\Im\tau}}\sum_{m=1}^{\infty}m^{k-1}q_{z}^{m}q_{\tau}^{-cm}+\frac{(+2\pi
i)^{k}}{(k-1)!}\sum_{c>\frac{\Im
z}{\Im\tau}}\sum_{m=1}^{\infty}m^{k-1}q_{\tau}^{cm}q_{z}^{-m}\] (recall the
value of $w$) plus a term appearing only if $\frac{\Im
z}{\Im\tau}\in\mathbb{Z}$. Replacing $m$ by $-m$ in case $c>\frac{\Im
z}{\Im\tau}$ and noticing that the sign coming from $m^{k-1}$ in this case and
the sign difference between the coefficients $(-2\pi i)^{k}$ and $(+2\pi i)^{k}$
combine just to $-1$ allows us to prove

\begin{prop}
The function $\frac{(-1)^{k}}{(k-1)!}\wp_{\tau}^{(k-2)}(z)$ equals
\[\delta\cdot(\zeta(k,a)+(-1)^{k}\zeta(k,1-a))+\frac{(-2\pi
i)^{k}}{(k-1)!}\sum_{c\neq\frac{\Im
z}{\Im\tau}}\sum_{m\rho>0}sgn(m)m^{k-1}q_{z}^{m}q_{\tau}^{-cm},\] where
$\rho=\Im(z-c\tau)$. The coefficient $\delta$ equals 1 if $\Im z$ is an integral
multiple $c$ of $\Im\tau$ (and then $a=\{z-c\tau\}$ with
$z-c\tau\in\mathbb{R}\setminus\mathbb{Z}$), and vanishes otherwise.
\label{pkexp}
\end{prop}

The idea is that the expression in Proposition \ref{pkexp} shows that
$\wp_{\tau}^{(k-2)}$ is a lattice function. Indeed, we claim that both parts of
this expression are invariant under both translations $z \mapsto z+1$ and $z
\mapsto z+\tau$. $q_{z}$ and $\delta$ are invariant under $z \mapsto z+1$
(the latter depends only on $\Im z$). Moreover, for $\delta=1$ we use the same
$c$ so that the argument of the fractional value $a$ is changed by an integer.
Hence $a$ is also invariant under this translation, and so is the expression in
Proposition \ref{pkexp}. As for $z \mapsto z+\tau$, it takes $\rho=\Im(z-c\tau)$
to $\Im(z-(c-1)\tau)$, $q_{z}^{m}q_{\tau}^{-cm}$ to
$q_{z}^{m}q_{\tau}^{-(c-1)m}$, and the condition $c\neq\frac{\Im z}{\Im\tau}$ to
$c\neq\frac{\Im z}{\Im\tau}+1$. Moreover, the quotients $\frac{\Im z}{\Im\tau}$
and $\frac{\Im(z+\tau)}{\Im\tau}$ differ by an additive integer, hence are
integral together and the value of $\delta$ is invariant under this
translation. In addition, if $\frac{\Im z}{\Im\tau}$ is an integer $c$,
then the integer $\frac{\Im(z+\tau)}{\Im\tau}$ is $c+1$, and the real numbers
$z+\tau-(c+1)\tau$ and $z-c\tau$ coincide, whence the invariance of $a$.
Therefore replacing $c$ by $c+1$ throughout shows the invariance of the
expression in Proposition \ref{pkexp} under this translation as well. This shows
the ellipticity of $\wp_{\tau}^{(k-2)}$, and the homogeneity property of
$\wp^{(k-2)}$, namely \[\wp_{\alpha L}^{(k-2)}(\alpha
z)=\alpha^{-k}\wp_{L}^{(k-2)}(z)\] for every lattice $L\subseteq\mathbb{C}$,
$z\in\mathbb{C}$, and $0\neq\alpha\in\mathbb{C}$, shows that $\wp_{L}^{(k-2)}$
is an elliptic function for any lattice $L$. Proposition \ref{pkexp} with $k=3$
is related to the Fourier expansion of $\wp_{\tau}'(z)$ given in Proposition 3
of Section 2 in Chapter 4 of \cite{[L]}, from which one can also obtain the
ellipticity of $\wp_{\tau}'$ (hence of $\wp_{L}'$ for any $L$ by homogeneity).

\medskip

For $\wp_{L}^{(k-2)}$ we do not need these considerations, since its ellipticity
is clear from its defining series. However, for $\wp$ itself this property is
not so obvious. Indeed, the ellipticity of $\wp$ follows from that of $\wp'$ and
the fact that $\wp$ is an even function of $z$, but it is expedient to see
(before we move on to the more complicated case of $Z$) how it can also be
obtained using our argument. Specializing the definition of $\wp_{L}(z)$ to
$L=L_{\tau}$ we obtain
\[\wp_{\tau}(z)=\frac{1}{z^{2}}+\sum_{(c,d)\neq(0,0)}\bigg(\frac{1}{
(z-(c\tau+d))^{2}}-\frac{1}{(c\tau+d)^{2}}\bigg),\] and again we carry out the
summation first over $d$ and then over $c$. Both parts of the sum converge, and
we get
\[\wp_{\tau}(z)=\sum_{c\in\mathbb{Z}}\sum_{d\in\mathbb{Z}}\frac{1}{
(z-(c\tau+d))^{2}}-\sum_{c\in\mathbb{Z}}\sum_{d,(c,d)\neq(0,0)}\frac{1}{
(c\tau+d)^{2}}.\] The second sum is (by definition) the classical Eisenstein
series $G_{2}(\tau)$ (and is independent of $z$), and the first one can be
evaluated as in the derivation of Proposition \ref{pkexp}. We thus obtain

\begin{prop}
The function $\wp_{\tau}(z)$ equals
\[\delta\cdot(\zeta(2,a)+\zeta(2,1-a))+(-2\pi i)^{2}\sum_{c\neq\frac{\Im
z}{\Im\tau}}\sum_{m\rho>0}sgn(m)mq_{z}^{m}q_{\tau}^{-cm}-G_{2}(\tau),\] with
$\rho$, $\delta$, and $a$ bearing the same meaning as in Proposition
\ref{pkexp}. \label{p2exp}
\end{prop}

The same argument used for the derivatives of $\wp_{\tau}$ yields the
ellipticity of $\wp_{\tau}$, since the additional term $G_{2}(\tau)$ is a
constant, i.e., independent of $z$. The homogeneity property
\[\wp_{\alpha L}(\alpha z)=\alpha ^{-2}\wp_{L}(z)\] for $L$, $z$, and $\alpha$
as above shows the ellipticity of $\wp_{L}$ for any $L$, without using the
parity of $\wp_{L}$. We mention the existence of other forms of Fourier
expansion of $\wp_{\tau}(z)$---see, for example Propositions 2 and 3 of Section
2 in Chapter 4 of \cite{[L]}. From some of them (like the latter example here)
the ellipticity of $\wp_{\tau}$ (thus, by homogeneity, also of $\wp_{L}$ for any
$L$) can be deduced as with our argument.

The expression for $G_{2}(\tau)$ is seen to be convergent by precisely the same
method, yielding (using some symmetry conditions) the well-known Fourier
expansion
\[G_{2}(\tau)=2\zeta(2)+2(2\pi
i)^{2}\sum_{c=1}^{\infty}\sum_{m=1}^{\infty}mq_{\tau}^{cm}=\frac{\pi^{2}}{3}
-8\pi^{2}\sum_{n}\sigma_{1}(n)q_{\tau}^{n}\] with $\sigma_{1}(n)=\sum_{d|n}d$
(see, for example, Chapter 1 of \cite{[DS]}). The convergence properties of this
expansion show that $G_{2}$ is holomorphic (and also invariant under
$\tau\mapsto\tau+1$), but no further properties are deduced at this point.

\section{The Weierstrass Zeta Function}\label{Zeta}

The above discussion serves well to illustrate how our method works, but it
offers no obvious advantage because the standard derivation is just as simple.
Here we demonstrate the usefulness of the technique by applying it to evaluate
$\eta_{\tau}$ directly (saving the contour integrals). The Legendre relation
becomes just a simple corollary, rather than a tool in the proof.

We specialize $Z_{L}(z)$ to $L=L_{\tau}$ and obtain
\begin{equation}
Z_{\tau}(z)=\frac{1}{z}+\sum_{(c,d)\neq(0,0)}\bigg(\frac{1}{z-(c\tau+d)}+\frac
{1}{c\tau+d}+\frac{z}{(c\tau+d)^{2}}\bigg) \label{Zetatau}
\end{equation}
(which is known to be an absolutely convergent sum), and fix the following
summation order: We sum first over $d$ and then over $c$, and for each index
($c$ or $d$) we take first the value 0, and then the terms with $d$ and $-d$ (in
the inner sum) or $c$ and $-c$ (in the outer sum) together. We then write
$Z_{\tau}(z)$ as the sum of three terms, each corresponding to one of the
elements in the parentheses of Equation \eqref{Zetatau}. The first term becomes
\[\frac{1}{z}+\sum_{f=1}^{\infty}\bigg(\frac{1}{z+f}+\frac{1}{z-f}\bigg)+\sum_{
e=1}^{\infty}\Bigg(\frac{1}{z+e\tau}+\sum_{f=1}^{\infty}\bigg(\frac{1}{z+e\tau+f
}+\frac{1}{z+e\tau-f}\bigg)+\]
\begin{equation}
+\frac{1}{z-e\tau}+\sum_{f=1}^{\infty}\bigg(\frac{1}{z-e\tau+f}+\frac{1}{
z-e\tau-f}\bigg)\Bigg),
\label{Zetaord}
\end{equation}
and similar expressions are obtained for the second and third terms.

We claim that the series in all three terms converge (in this order of
summation). Indeed, the second term begins with
$\sum_{f=1}^{\infty}\big(\frac{1}{f}+\frac{1}{-f}\big)=0$ and continues with
\[\sum_{e=1}^{\infty}\Bigg(\frac{1}{e\tau}+\sum_{f=1}^{\infty}\bigg(\frac{1}{
e\tau+f}+\frac{1}{e\tau-f}\bigg)+\frac{1}{-e\tau}+\sum_{f=1}^{\infty}\bigg(\frac
{1}{-e\tau+f}+\frac{1}{-e\tau-f}\bigg)\Bigg)\] which also vanishes. The third
term yields directly $zG_{2}(\tau)$. Therefore the series in Equation
\eqref{Zetaord} converges to $Z_{\tau}(z)-zG_{2}(\tau)$, so that evaluating it
explicitly gives an expression for $Z_{\tau}(z)$.

\medskip

Let us perform the evaluation. Note that the first part of Equation
\eqref{Zetaord} is just the left hand side of Equation \eqref{ctgH},
\eqref{ctgHbar} or \eqref{ctgR} for $w=z$, and each $e$ gives the sum of the
left hand side of Equation \eqref{ctgH}, \eqref{ctgHbar} or \eqref{ctgR} with
$w=z+e\tau$ and with $w=z-e\tau$. We thus substitute the right hand side of the
corresponding equation (and value of $w$), and obtain expressions similar to
those encountered in the proof of Propositions \ref{pkexp} and \ref{p2exp}.
However, there are two differences. First, for integral $\frac{\Im z}{\Im\tau}$
we now have $\pi\cot\pi a$ rather than the Hurwitz zeta function. Second, and
more important, are the constants $-\pi i$ in Equation \eqref{ctgH} and $+\pi i$
in Equation \eqref{ctgHbar}, which should be added to the power series in
$\mathbf{e}(\pm w)$. The part involving $q_{z}$ and $q_{\tau}$ converges as with
$\wp_{\tau}$ and its derivatives. It remains to verify that the constants also
converge in the chosen order of summation. This follows from the fact that for
$e>\frac{|\Im z|}{\Im\tau}$, $z+e\tau$ lies in $\mathcal{H}$ and $z-e\tau$ lies
in $\overline{\mathcal{H}}$, so that we use Equation \eqref{ctgH} for the first
and Equation \eqref{ctgHbar} for the second, and the constants cancel. Hence for
all but finitely many values of $e$ the constants cancel, and the others give
rise to some integral multiple of $-\pi i$. This proves

\begin{prop}
The equality \[Z_{\tau}(z)=-t\pi i+\delta\cdot\pi\cot\pi a-2\pi
i\sum_{c\neq\frac{\Im
z}{\Im\tau}}\sum_{m\rho>0}sgn(m)q_{z}^{m}q_{\tau}^{-cm}+zG_{2}(\tau)\] holds,
where $\rho$, $\delta$, and $a$ have the same meaning as in Propositions
\ref{pkexp} and \ref{p2exp}, and $t$ is some (finite) integer depending on
$\tau$ and $z$. \label{p1exp}
\end{prop}

Indeed, as we have seen, a finite number of $e$-summands (possibly together with
the one corresponding to $e=0$) involve a constant $\mp\pi i$ from Equations
\eqref{ctgH} and \eqref{ctgHbar}. Recall that the contribution $zG_{2}(\tau)$ of
the third summands in Equation \eqref{Zetatau} must also be included. We remark
that Proposition \ref{p2exp} can be obtained directly from Proposition
\ref{p1exp} by differentiation with respect to $z$. This step requires some
care: One should differentiate with respect to the real part, noting that
$\delta$ depends only on $\Im z$ and, as Proposition \ref{tval} below shows, the
same holds for $t$---this explains why $\delta$ remains invariant and the term
with $t$ vanishes after taking the derivative. Similarly, one can deduce
Proposition \ref{pkexp} by a $(k-2)$-fold differentiation of the result of
Proposition \ref{p2exp} with respect to $z$. A Fourier expansion for
$Z_{\tau}(z)$, up to a factor which is linear in $z$, appears in Equation (1) in
Section 3 of Chapter 18 of \cite{[L]} (and another formula appearing right after
it), but this linear factor is written in terms of $\eta_{\tau}(1)$, hence
cannot be used alone for the evaluation of $\eta_{\tau}(1)$.

\medskip

It remains to find the value of the integer $t$, where we recall that only
$e\leq\frac{|\Im z|}{\Im\tau}$ should be considered in Equation \eqref{Zetaord}.
For $z\not\in\mathbb{R}$ the first sum (which does not involve $e$) contributes
$sgn(\Im z)$ to $t$, while for $z\in\mathbb{R}$ there is no contribution at all.
Indeed, for $\Im z>0$ we use Equation \eqref{ctgH} with the constant $-\pi i$,
for $\Im z<0$ we employ Equation \eqref{ctgHbar} with $+\pi i$ (recall that $t$
is the coefficient of $-\pi i$), while Equation \eqref{ctgR}, used for
$z\in\mathbb{R}$, provides no contribution to $t$. For every $e<\frac{|\Im
z|}{\Im\tau}$ (this can be an empty set of integers, as is the case where
$-\Im\tau\leq\Im z\leq\Im\tau$) both $z+e\tau$ and $z-e\tau$ have the same sign
of imaginary part as $z$, so that each such $e$ contributes $2sgn(\Im z)$ to
$t$. In the case where $\frac{\Im z}{\Im\tau}$ is a non-zero integer we see that
for the value $e=\frac{|\Im z|}{\Im\tau}$ one summand gives real $w$ (and no
contribution to the constant) and the other gives a contribution of $sgn(\Im z)$
to $t$ as above. This is the basis of the proof of

\begin{prop}
The integer $t$ from Proposition \ref{p1exp} is given by
\[t=\bigg\lfloor\frac{\Im z}{\Im\tau}\bigg\rfloor+\bigg\lceil\frac{\Im
z}{\Im\tau}\bigg\rceil.\]
\label{tval}
\end{prop}

\begin{proof}
Write $x=\frac{\Im z}{\Im\tau}$, and we begin by assuming the $x$ is not an
integer. Then the number of $e$-summands which have a contribution of 2 is
$\lfloor|x|\rfloor$, so that the above argument yields
$t=(2\lfloor|x|\rfloor+1)sgn(x)$. For positive $x$ this is just $2\lfloor x
\rfloor+1=\lfloor x \rfloor+\lceil x \rceil$ ($x$ is not an integer, so that
$\lceil x \rceil=\lfloor x \rfloor+1$). For negative $x$ we write
$\lfloor-x\rfloor=-\lceil x \rceil$, so that $-2\lfloor-x\rfloor-1=2\lceil x
\rceil-1=\lfloor x \rfloor+\lceil x \rceil$ (we still assume
$x\not\in\mathbb{Z}$). This shows that $t$ has the asserted expression in this
case. Next consider the case $0 \neq x\in\mathbb{Z}$ is a nonzero integer.
In this case the number of $e$-summands contributing 2 is $|x|-1$ (since this is
defined by a sharp inequality on $e$). An extra contribution of one comes from
$z$, and another one comes from $e=|x|$, yielding $t=2|x|sgn(x)=2x$. Since for
integral $x$ we have $\lfloor x \rfloor=\lceil x \rceil=x$ this also agrees with
the asserted expression. For $x=0$ (i.e., real $z$) we have no constant
contribution at all, and the value $t=0$ is indeed the asserted expression for
$x=0$. This covers all the possible cases, hence proves the proposition.
\end{proof}

Applying the identity $\lfloor-x\rfloor=-\lceil x \rceil$ again shows that
$t=\big\lfloor\frac{\Im z}{\Im\tau}\big\rfloor-\big\lfloor\frac{-\Im
z}{\Im\tau}\big\rfloor$, which has the advantage of using only the (more
intuitive) lower integral value function. Moreover, it reflects better the fact
that $Z_{\tau}$ is an odd function of $z$. However, only the equality in
Proposition \ref{tval} will be used below. 

\medskip

We can now evaluate the lattice function $\eta_{\tau}$ directly from
Propositions \ref{p1exp} and \ref{tval}, as given in the following

\begin{thm}
The difference function $\eta_{\tau}$ is given on the generators 1 and $\tau$ of
$L_{\tau}$ by \[\eta_{\tau}(1)=G_{2}(\tau),\qquad \eta_{\tau}(\tau)=\tau
G_{2}(\tau)-2\pi i.\] For a general element $\lambda=c\tau+d$ of $L_{\tau}$ we
have \[\eta_{\tau}(c\tau+d)=(c\tau+d)G_{2}(\tau)-2\pi ic.\] \label{diffunc}
\end{thm}

\begin{proof}
We know that $t$, $\delta$, and also $a$ if $\delta=1$, depend only in $\Im z$,
and the sum in Proposition \ref{p1exp} depends on $z$ only through $q_{z}$. It
follows that the translation $z \mapsto z+1$ adds to $Z_{\tau}(z)$ only
$G_{2}(\tau)$ from the last term. This proves the value of $\eta_{\tau}(1)$.
Turning to the translation $z \mapsto z+\tau$, the sum is invariant by the usual
summation index change argument, and the term with $\delta$ and $a$ is also
invariant as above. On the other hand, $t$ is increased by 2. Indeed, $\frac{\Im
z}{\Im\tau}$ increases by 1, hence so do the lower and upper integral values.
Taking also the last summand into consideration, we obtain the asserted value of
$\eta_{\tau}(\tau)$. The expression for the general value $\eta_{\tau}(c\tau+d)$
can be obtained the additivity of $\eta_{\tau}$. Alternatively, we can deduce
the general value directly from Propositions \ref{p1exp} and \ref{tval}: The
term with $G_{2}(\tau)$ in Proposition \ref{p1exp} gives the first summand, the
sum and $\delta$ (and $a$) are invariant under any change $z \mapsto z+\lambda$
for $\lambda \in L_{\tau}$, and the value of $t$ is increased by $2c$. In this
approach, the values of $\eta_{\tau}(1)$ and $\eta_{\tau}(\tau)$ are just
special cases of the general formula.
\end{proof}

We recall the homogeneity property of $Z$, which implies a similar property for
$\eta$, namely \[Z_{\alpha L}(\alpha z)=\alpha^{-1}Z_{L}(z),\qquad\eta_{\alpha
L}(\alpha\lambda)=\alpha^{-1}\eta_{L}(\lambda)\] for $L$, $z$, and $\alpha$ as
above and $\lambda \in L$. Combining this with Theorem \ref{diffunc} implies
that for a general lattice $L=\mathbb{Z}w_{1}\oplus\mathbb{Z}w_{2}$ normalized
such that $\tau=\frac{w_{1}}{w_{2}}$ is in $\mathcal{H}$, we have
\[\eta_{L}(w_{2})=\frac{G_{2}(\tau)}{w_{2}},\qquad\eta_{L}(w_{1})=\frac{\tau
G_{2}(\tau)-2\pi i}{w_{2}}.\] This is so since $L=w_{2}L_{\tau}$. We therefore
obtain

\begin{cor}
The Legendre relation holds: \[w_{1}\eta_{L}(w_{2})-w_{2}\eta_{L}(w_{1})=+2\pi
i.\] \label{Legendre}
\end{cor}

Substituting the value of $\tau$ in the expressions for $\eta_{L}(w_{1})$ and
$\eta_{L}(w_{2})$ proves Corollary \ref{Legendre} immediately. Alternatively,
Corollary \ref{Legendre} for the special case of the lattice $L=L_{\tau}$ with
the basis $w_{2}=1$ and $w_{1}=\tau$ follows immediately from Theorem
\ref{diffunc}, and then the homogeneity property of $\eta_{L}$, compensated by
the (trivial) homogeneity of the coefficients, extends the validity of Corollary
\ref{Legendre} to any lattice $L$ with a normalized basis. We remark that in
\cite{[L]} (as well as in other references dealing with the function
$\eta_{\tau}$), one first uses integration in order to obtain the Legendre
relation in Corollary \ref{Legendre}, then one evaluates (by more difficult
means) $\eta_{\tau}(1)$, and only then the value of $\eta_{\tau}(\tau)$ (hence
of $\eta_{\tau}(\lambda)$ for any $\lambda \in L_{\tau}$) can be obtained. In
our approach, one derives all the values of $\eta_{\tau}(\lambda)$ at once, and
Corollary \ref{Legendre} is established by a simple substitution.

\section{Quasi-Modularity of $G_{2}$}\label{G2}

At this point, all we know about the Eisenstein series $G_{2}$ is that it is
holomorphic and invariant under $\tau\mapsto\tau+1$. We now use its relation
with $\eta_{\tau}$ and the homogeneity of the latter to obtain its quasi-modular
behavior under the action of $SL_{2}(\mathbb{Z})$.

\begin{thm}
If the matrix $M=\binom{a\ \ b}{c\ \ d}$ is in $SL_{2}(\mathbb{Z})$ then we have
\[G_{2}\bigg(\frac{a\tau+b}{c\tau+d}\bigg)=(c\tau+d)^{2}G_{2}(\tau)-2\pi
ic(c\tau+d).\] \label{G2qmod}
\end{thm}

\begin{proof}
The complex numbers $w_{1}=a\tau+b$ and $w_{2}=c\tau+d$ are clearly elements of
$L_{\tau}$, and the generators $\tau$ and 1 of $L_{\tau}$ can be written as
$dw_{1}-bw_{2}$ and $aw_{2}-cw_{1}$ respectively (since $ad-bc=1$). Moreover,
the quotient $\frac{w_{1}}{w_{2}}=\frac{a\tau+b}{c\tau+d}$ (which we write as
$M\tau$) has imaginary part $\frac{\Im\tau}{|c\tau+d|^{2}}$ (again since
$ad-bc=1$). Hence $L_{\tau}=\mathbb{Z}w_{1}\oplus\mathbb{Z}w_{2}$, and this is a
normalized basis, and $M\tau\in\mathcal{H}$. We then write $L_{\tau}$ as
$(c\tau+d)L_{M\tau}$, and using the general formula in Theorem \ref{diffunc} and
the homogeneity property of $\eta_{L}$ we obtain the equality
\[(c\tau+d)G_{2}(\tau)-2\pi
ic=\eta_{\tau}(c\tau+d)=\frac{\eta_{M\tau}(1)}{c\tau+d}=\frac{G_{2}(M\tau)}{
c\tau+d}.\] After multiplication by $c\tau+d$, this proves the theorem.
\end{proof}

Note that this proof does not require knowing that if the asserted relation
holds for two matrices $M$ and $N$ then it holds for their product. Moreover, it
is independent of the fact that $\binom{1\ \ 1}{0\ \ 1}$ and $\binom{0\ \ -1}{1\
\ \ \ 0}$ generate $SL_{2}(\mathbb{Z})$, and even of the fact that
$(M,\tau)\mapsto\frac{a\tau+b}{c\tau+d}$ (for $M$ as above) defines an action of
$SL_{2}(\mathbb{Z})$ on $\mathcal{H}$. The invariance under $\tau\mapsto\tau+1$
already mentioned above is just the case $M=\binom{1\ \ 1}{0\ \ 1}$ in Theorem
\ref{G2qmod}.

\noindent\textsc{Fachbereich Mathematik, AG 5, Technische Universit\"{a}t
Darmstadt, Schlossgartenstrasse 7, D-64289, Darmstadt, Germany}

\noindent E-mail address: zemel@mathematik.tu-darmstadt.de


\begin{thebibliography}{}{}

\bibitem[DS]{[DS]} Diamond, F., Shurman, J., \textsc{A First Course in Modular
Forms}, Graduate Texts in Mathematics 228, Springer-Verlag, New York (2005).
\bibitem[L]{[L]} Lang, S., \textsc{Elliptic Functions}, 2nd Edition, Graduate
Texts in Mathematics 112, Springer-Verlag, New York (1987).

\end{thebibliography}
\end{document}